\documentclass{amsart}
\usepackage{amsfonts}
\usepackage{amsmath,amssymb}
\usepackage{amsthm}
\usepackage{amscd}
\usepackage{graphics}
\usepackage{graphicx}

\theoremstyle{remark}{
\newtheorem{Def}{{\rm Definition}}

\newtheorem{Rem}{{\rm Remark}}
\newtheorem{Prob}{{\rm Problem}}

}
\theoremstyle{plain}
{

\newtheorem{Prop}{Proposition}

\newtheorem{MainThm}{Main Theorem}

}

\begin{document}
\title[Non-proper smooth real algebraic functions on smooth manifolds]{Explicit smooth real algebraic functions which may have both compact and non-compact preimages on non-singular real algebraic manifolds}
\author{Naoki kitazawa}
\keywords{Real algebraic functions and maps. Real algebraic sets. Real algebraic hypersurfaces. Hypersurface arrangements. Reeb graphs. \\
\indent {\it \textup{2020} Mathematics Subject Classification}: Primary~14P05, 14P15, 14P25, 57R45, 58C05. Secondary~52C17, 57R19.}
\address{Osaka Central
	Advanced Mathematical Institute, 3-3-138 Sugimoto, Sumiyoshi-ku Osaka 558-8585 \\
	TEL (Office): +81-6-6605-3103 \\
	FAX (Office): +81-6-6605-3104 \\
}
\email{naokikitazawa.formath@gmail.com}
\urladdr{https://naokikitazawa.github.io/NaokiKitazawa.html}
\maketitle
\begin{abstract}
In our previous preprint, which is withdrawn due to an improvement, we have constructed explicit smooth real algebraic functions which may have both compact and non-compact preimages on {\it non-singular} real algebraic manifolds. This paper presents its variant. Our result is new in obtaining non-proper smooth real algebraic functions on non-singular real algebraic manifolds satisfying explicit conditions on (non-)compactness of preimages whereas previously the manifolds are only {\it semi-algebraic}.

Explicitly, this mainly contributes to two different regions. One is singularity theory of differentiable maps and applications to differential topology. More precisely, construction of nice smooth maps with desired preimages. The other is real algebraic geometry. More precisely, explicit construction of smooth real algebraic functions and maps: we can know the existence and consider approximations of smooth maps by maps of such classes in considerable cases.

\end{abstract}
\section{Introduction.}
\label{sec:1}

Our paper is on new explicit construction of smooth real algebraic functions and maps on non-singular real algebraic manifolds. This is also a variant \cite{kitazawa7}, which is withdrawn due to an improvement. Presentations of \cite{kitazawa7} and our present paper are similar in considerable scenes.

\subsection{Two interest related to our study: one is from singularity theory of differentiable maps and its applications to differential topology and the other is from real algebraic geometry.}
Related to our study, we present a study on singularity theory of differentiable maps and applications to differential topology.
Previously, the author has constructed smooth functions with prescribed preimages explicitly in \cite{kitazawa1, kitazawa2, kitazawa5}. These are answers to a problem of the author, a revised version of Sharko's problem. Sharko asks whether we can obtain smooth functions whose {\it Reeb graphs} are as desired. 
\cite{masumotosaeki, michalak} are among explicit answers to Sharko's original question and they have also motivated the author to obtain the results.
The {\it Reeb graph} of a smooth function is a graph obtained as the space of connected components of preimages of smooth functions of some nice wide class. Such a class is studied in \cite{saeki} for example. Reeb graphs are classical objects already in \cite{reeb}. They inherit several important topological information of the manifolds.

We are also interested in real algebraic geometry. According to general theory, pioneered by Nash for example, smooth manifolds are so-called {\it non-singular} real algebraic manifolds and we can approximate smooth maps by smooth real algebraic maps and know the existence in considerable cases. \cite{akbulutking, bochnakcosteroy, kollar, kucharz, nash, shiota, tognoli} explain about this for example.
Our interest lies in explicit construction of explicit smooth real algebraic functions or maps of non-positive codimensions. Natural projections of spheres embedded naturally in the one-dimensional higher Euclidean spaces are simplest examples on closed manifolds. As functions regarded as generalized cases, see \cite{maciasvirgospereirasaez, ramanujam, takeuchi} for example.
They are given by suitable real polynomials. 

Under these backgrounds, we consider the following problem.
\begin{Prob}
On the other hands, as natural questions on geometry, explicit global structures of these real algebraic functions and properties are hard to understand. For example, can we know about preimages? As another example, can we know about the real polynomials for our desired zero sets and real algebraic manifolds? 
\end{Prob}
\cite{kitazawa3} is one of our pioneering study. \cite{kitazawa4} is the abstract of our related talk in a conference. The work is followed by us in \cite{kitazawa6, kitazawa7, kitazawa8}. Our new study is on smooth functions which may be non-proper whereas these studies show explicit construction of real algebraic functions on non-singular real algebraic closed manifolds.

\subsection{Manifolds and maps.}
Let $X$ be a topological space homeomorphic to some cell complex of a finite dimension. We can define its dimension $\dim X$ as a unique integer. A topological manifold is well-known to be homeomorphic to some CW complex. A smooth manifold is well-known to be homeomorphic to some polyhedron. We can also define the structure of a certain polyhedron for a smooth manifold in a canonical and unique way. This is a so-called PL manifold. We do not need to understand such theory precisely here.
  
Let ${\mathbb{R}}^k$ denote the $k$-dimensional Euclidean space, which is a simplest $k$-dimensional smooth manifold. This is also the Riemannian manifold endowed with the standard Euclidean metric. Let $\mathbb{R}:={\mathbb{R}}^1$ and $\mathbb{Z} \subset \mathbb{R}$ denote the set of all integers. Let $\mathbb{N} \subset \mathbb{Z}$ denote the set of all positive ones.
For each point $x \in {\mathbb{R}}^k$, let $||x|| \geq 0$ denote the distance between $x$ and the origin $0$ where the metric is the standard Euclidean metric.
This is also naturally a ({\it non-singular}) real algebraic manifold. This is also a so-called ({\it non-singular}) {\it Nash} manifold and a real analytic manifold. 
This real algebraic manifold is also the {\it $k$-dimensional real affine space}. Let $S^k:=\{x \in {\mathbb{R}}^{k+1} \mid ||x||=1\}$ denote the $k$-dimensional unit sphere. This
is a $k$-dimensional smooth compact submanifold of ${\mathbb{R}}^{k+1}$. It has no boundary. It is connected for any positive integer $k \geq 1$. It is a discrete set with exactly two points for $k=0$. It is a {\it non-singular} {\it real algebraic} set, which is the zero set of the real polynomial $||x||^2={\Sigma}_{j=1}^{k+1} {x_j}^2$ with $x:=(x_1,\cdots,x_{k+1})$.
Let $D^k:=\{x \in {\mathbb{R}}^{k} \mid ||x|| \leq 1\}$ denote the $k$-dimensional unit disk. It is a $k$-dimensional smooth compact and connected submanifold of ${\mathbb{R}}^{k}$ for any non-negative integer $k>0$. It is also a {\it semi-algebraic} set.

Let $c:X \rightarrow Y$ be a differentiable map from a differentiable manifold $X$ into another differentiable manifold $Y$. $x \in X$ is a {\it singular} point of the map if the rank of the differential at $x$ is smaller than the minimum between the dimensions $\dim X$ and $\dim Y$. We call $c(x)$ a {\it singular value} of $c$.
Let $S(c)$ denote the {\it singular set} of $c$. It is defined as the set of all singular points of $c$.
Unless otherwise stated, here, differentiable maps are smooth maps, which are maps of the class $C^{\infty}$.

A {\it canonical projection of the Euclidean space ${\mathbb{R}}^k$} is defined as the smooth surjective map mapping 
each point $x=(x_1,x_2) \in {\mathbb{R}}^{k_1} \times {\mathbb{R}}^{k_2}={\mathbb{R}}^k$ to the first component $x_1 \in {\mathbb{R}}^{k_1}$ with the conditions on the dimensions $k_1, k_2>0$ and $k=k_1+k_2$. Let ${\pi}_{k,k_1}:{\mathbb{R}}^{k} \rightarrow {\mathbb{R}}^{k_1}$ denote this map. We can define a {\it canonical projection of the unit sphere $S^{k-1}$} as the restriction of this there.

\subsection{Our main results.}


Let $X \subset \mathbb{R}$ be a subset. Let $a \in X$. Let $X_{a} \subset X$ denote the set of all elements of $X$ smaller than or equal to $a$.

The {\it real algebraic} set $X \in {\mathbb{R}}^k$ {\it defined by $l \geq 0$ real polynomials} is the intersection of the zero sets for these $l$ real polynomials. In the case $l=0$, it is ${\mathbb{R}}^k$.
In short, (non-singular) real algebraic manifolds are based on such sets and (non-singular) {\it Nash} manifolds are based on {\it semi-algebraic} sets, defined by considering not only the zero sets of real polynomials, but also the inequalities.

Our real algebraic manifolds are unions of connected components of some real algebraic sets defined by finitely many real polynomials unless otherwise stated. {\it Non-singular} real algebraic manifolds are defined by these polynomials. More precisely, such notions are defined by the ranks of the maps defined canonically from these polynomials with implicit function theorem.  
\begin{MainThm}
	\label{mthm:1}

	Let $l>1$ be an integer. Let $\{t_j\}_{j=1}^l$ be an increasing sequence of real numbers. Let $l_{{\mathbb{N}}_{l-1}}:{\mathbb{N}}_{l-1} \rightarrow \{0,1\}$ be a map. Let $m$ be a sufficiently large integer. Then we have a suitable $m$-dimensional non-singular real algebraic connected manifold $M$ with no boundary and a smooth real algebraic function $f:M \rightarrow \mathbb{R}$ enjoying the following properties.
	\begin{enumerate}
		\item $f(M)=[t_1,t_{l}]$ and the image $f(S(f))$ of the singular set of $f$ is $\{t_j\}_{j=1}^l$.
		\item $f^{-1}(t)$ is closed and connected if $t \in (t_{j},t_{j+1})$ and $l_{{\mathbb{N}}_{l-1}}(j)=0$. $f^{-1}(t)$ is connected and non-compact and it has no boundary if $t \in (t_{j},t_{j+1})$ and $l_{{\mathbb{N}}_{l-1}}(j)=1$. $f^{-1}(t)$ is connected and homeomorphic to a CW complex whose dimension is at most $m-1$ for any $t \in f(M)$. 
	\end{enumerate}
Furthermore, for a suitable integer $n>1$ and another sufficiently large integer $m>n$, we can suitably have a desired $m$-dimensional manifold $M$ and a desired function $f:M \rightarrow \mathbb{R}$ as the composition of a nice smooth real algebraic map on $M$ into ${\mathbb{R}}^n$ with a canonical projection. The nice smooth map into ${\mathbb{R}}^n$ is presented as a very explicit map in Main Theorem \ref{mthm:2}, later.

	\end{MainThm}

Main Theorem \ref{mthm:2} shows that the map into ${\mathbb{R}}^n$ can be constructed as one having a globally simple structure. Furthermore, this is very explicit and we can know the real polynomial for the definition of the real algebraic manifold (set). 
This is a variant of our main result of \cite{kitazawa8}, which is withdrawn due to an improvement. There we have constructed the manifolds as Nash ones and the functions and maps as the restrictions of smooth real algebraic functions and maps to semi-algebraic sets. Main Theorems here are also motivated by some differential topological viewpoint or \cite{kitazawa2} and there we have constructed smooth functions with prescribed preimages on non-compact manifolds with no boundaries. The functions of \cite{kitazawa2} are non-proper in general. We also construct smooth functions which are not real analytic there.


The next section explains about Main Theorems including our proofs. As our related previous work, we have constructed our functions and maps avoiding existence theory and approximations. The third section is a kind of appendices and related problems are presented.
\\
\ \\
\noindent {\bf Conflict of interest.} \\
The author was a member of the project JSPS Grant Number JP17H06128 and the project JSPS KAKENHI Grant Number JP22K18267 "Visualizing twists in data through monodromy" (Principal Investigator: Osamu Saeki). Their support has helped the author to do this study. 
The author is a researcher at Osaka Central
Advanced Mathematical Institute (OCAMI researcher). Note that he is not employed there. This is for our studies and helps our studies. \\
\ \\
{\bf Data availability.} \\
Data essentially supporting our present study are all in the
 paper.
\section{On Main Theorems.}
\label{sec:2}
\subsection{Important subsets in the real affine spaces.}
\label{subsec:2.1}
We introduce several subsets in the real affine spaces. They are important in proving Main Theorems.
\subsubsection{$L_{{\rm P},p_1,p_2}$.}
\label{subsubsec:2.1.1}
 For distinct two points $p_1,p_2 \in {\mathbb{R}}^2$, let $L_{{\rm P},p_1,p_2} \subset {\mathbb{R}}^2$ denote the straight line, which is uniquely defined and regarded as a copy of the $1$-dimensional real affine space embedded in the $2$-dimensional real affine space ${\mathbb{R}}^2$ by the canonically defined smooth real algebraic embedding. This is also an {\it affine subspace in ${\mathbb{R}}^2$}, defined later in Definition \ref{def:1}.

\subsubsection{$C_{{\rm H},(-\infty,a],b,c}$.}
\label{subsubsec:2.1.2}
 Let $a$, $b$ and $c>0$ be real numbers.
We consider the subset in ${\mathbb{R}}^2$ represented
as $\{(x_1,x_2) \in {\mathbb{R}}^2 \mid (x_1-a)(x_2-b)=c\}$. Let it be denoted by $C_{{\rm H},(-\infty,a],b,c}$. 
This is the real algebraic set defined by the single real polynomial and this is a $1$-dimensional non-singular real algebraic manifold. It consists of exactly two connected components. These connected components are denoted by $C_{{\rm H}_{+},(-\infty,a],b,c}:=\{(x_1,x_2) \in C_{{\rm H},(-\infty,a],b,c} \mid x_1>a, x_2>b\}$ and $C_{{\rm H}_{-},(-\infty,a],b,c}:=\{(x_1,x_2) \in C_{{\rm H},(-\infty,a],b,c} \mid x_1<a, x_2<b\}$, respectively. Last, this is for a hyperbola.
\subsubsection{$C_{{\rm H},[a,\infty),b,c}$.} \label{subsubsec:2.1.3} Let $a$, $b$ and $c<0$ be real numbers.
We consider the subset in ${\mathbb{R}}^2$ represented
as $\{(x_1,x_2) \in {\mathbb{R}}^2 \mid (x_1-a)(x_2-b)=c\}$. Let it be denoted by $C_{{\rm H},[a,\infty),b,c}$. This is the real algebraic set defined by the single real polynomial and this is a $1$-dimensional non-singular real algebraic manifold. It consists of exactly two connected components. These connected components are denoted by $C_{{\rm H}_{+},[a,\infty),b,c}:=\{(x_1,x_2) \in C_{{\rm H},[a,\infty),b,c} \mid x_1<a, x_2>b\}$ and $C_{{\rm H}_{-},[a,\infty),b,c}:=\{(x_1,x_2) \in C_{{\rm H},[a,\infty),b,c} \mid x_1>a, x_2<b\}$, respectively. Last, this is also for a hyperbola.
\subsubsection{$B_{{\rm E},\{a_j\}_{j=1}^k,\{r_j\}_{j=1}^k}$.} 
\label{subsubsec:2.1.4}
 Let $\{a_j\}_{j=1}^k$ be a sequence of real numbers. Let $\{r_j\}_{j=1}^k$ be a sequence of positive numbers.
We consider the subset in ${\mathbb{R}}^k$ represented
as $\{(x_1,\cdots x_k) \in {\mathbb{R}}^k \mid {\Sigma}_{j=1}^k \frac{{(x_j-a_j)}^2}{r_j} \leq 1\}$. Let it be denoted by $B_{{\rm E},\{a_j\}_{j=1}^k,\{r_j\}_{j=1}^k}$. 
This is a $k$-dimensional smooth manifold diffeomorphic to the unit disk $D^k$. This is also a semi-algebraic set.
Its boundary is a ($k-1$)-dimensional non-singular real algebraic manifold and the real algebraic set defined by the single real polynomial ${\Sigma}_{j=1}^k \frac{{(x_j-a_j)}^2}{r_j}-1$. This is diffeomorphic to $S^{k-1}$ of course. This is for a $k$-dimensional ellipsoid and its boundary.
\subsubsection{Affine transformations and some classes of sets.}
\label{subsubsec:2.1.5}
Respecting these subsets for example, we define some important subsets. 
We also regard ${\mathbb{R}}^k$ as the $k$-dimensional real vector space in the canonical way.
A linear transformation on a linear space means a linear map which is also a linear isomorphism there.
An {\it affine transformation on the real affine space ${\mathbb{R}}^k$} is a map represented as the composition of a linear transformation on ${\mathbb{R}}^k$ with an operation considering the sum with a fixed value $a \in {\mathbb{R}}^k$ or the operation mapping $x$ to $x+a$. 
\begin{Def}
	\label{def:1}
	\begin{enumerate}
		\item An {\it affine subspace in ${\mathbb{R}}^k$} is ${\mathbb{R}}^k$ itself or a non-singular real algebraic submanifold embedded by the smooth real algebraic embedding mapping $x \in {\mathbb{R}}^{k_1}$ to $(x,0) \in {\mathbb{R}}^{k_1} \times {\mathbb{R}}^{k_2}={\mathbb{R}}^k$ with the conditions $k_1,k_2>0$ and $k_1+k_2=k$. It is also an {\it affine subspace in ${\mathbb{R}}^k$} if it is mapped onto such a subspace by some affine transformation on the real affine space ${\mathbb{R}}^k$.  
		\item A {\it {\rm (}$k-1${\rm )}-dimensional cylinder of a hyperbola in ${\mathbb{R}}^k$} is a non-singular real algebraic manifold and it is also the real algebraic set defined by a naturally defined single real polynomial and represented as a subset in \ref{subsubsec:2.1.2} or \ref{subsubsec:2.1.3} in the case $k=2$ or a form of the product of a subset of \ref{subsubsec:2.1.2} or \ref{subsubsec:2.1.3} in ${\mathbb{R}}^2$ and ${\mathbb{R}}^{k-2}$ in the case $k \geq 3$. A non-singular real algebraic submanifold mapped onto such a subspace by some affine transformation on the real affine space ${\mathbb{R}}^k$ is also called in this way.
		 
		\item A {\it $k$-dimensional ellipsoid in ${\mathbb{R}}^k$} is $B_{{\rm E},\{a_j\}_{j=1}^k,\{r_j\}_{j=1}^k}$ or a semi-algebraic set mapped onto such a subspace by some affine transformation on the real affine space ${\mathbb{R}}^k$.
		
		\item A {\it $k$-dimensional cylinder of some ellipsoid in ${\mathbb{R}}^k$} is $B_{{\rm E},\{a_j\}_{j=1}^{k^{\prime}},\{r_j\}_{j=1}^{k^{\prime}}} \times {\mathbb{R}}^{k-k^{\prime}} \subset {\mathbb{R}}^{k^{\prime}} \times {\mathbb{R}}^{k-k^{\prime}}$ with some integer satisfying $1 \leq k^{\prime}<k$ or a semi-algebraic set mapped onto such a subspace by some affine transformation on the real affine space ${\mathbb{R}}^k$.
	\end{enumerate}
\end{Def}
\subsection{A proposition for our proof of Main Theorems.}
For one of our main ingredients in our paper, we review construction of smooth real algebraic maps first presented in \cite{kitazawa3} and followed by \cite{kitazawa7, kitazawa8} for example.
For a finite set $X$, let $|X| \in \mathbb{N} \bigcup \{0\}$ denote its size.
\begin{Def}
\label{def:2}
Let $D$ be a connected open set in ${\mathbb{R}}^n$ enjoying the following properties. 
\begin{enumerate}
\item There exists a positive integer $l>0$.
\item There exists a family $\{f_j\}_{j=1}^l$ of $l$ real polynomials with $n$ variables. 
\item $D:={\bigcap}_{j=1}^l \{x \in {\mathbb{R}}^n \mid f_j(x)>0\}$ and for the closure $\overline{D}$ of $D$, $\overline{D}={\bigcap}_{j=1}^l \{x \in {\mathbb{R}}^n \mid f_j(x) \geq 0\}$. They are also semi-algebraic sets.
\item $S_j$ is a connected component of the real algebraic set $\{x \in {\mathbb{R}}^n \mid f_j(x)=0\}$ defined by $f_j(x)$ and a non-singular connected real algebraic hypersurface.
\item $(\{x \in {\mathbb{R}}^n \mid f_j(x)=0\}-S_j) \bigcap \overline{D}$ is empty and $\overline{D}$ is in ${\mathbb{R}}^n-(\{x \in {\mathbb{R}}^n \mid f_j(x)=0\}-S_j)$, which is an open set.
\item Distinct hypersurfaces in the family $\{S_j\}$ intersect satisfying the following conditions. In other words, "transversality" is satisfied. 
\begin{enumerate}
\item Each intersection ${\bigcap}_{j \in \Lambda} S_j$ ($\Lambda \subset {\mathbb{N}}_{l}$) is empty or a non-singular real algebraic hypersurface with no boundary and of dimension $n-|\Lambda|$ for each non-empty set $\Lambda$.
\item Let $e_j:S_j \rightarrow {\mathbb{R}}^n$ denote the canonically defined smooth real algebraic embedding. For each intersection ${\bigcap}_{j \in \Lambda} S_j$ as before and each point $p$ there, the following two agree.
\begin{enumerate}
	\item The image of the differential of the canonically defined embedding ${\bigcap}_{j \in \Lambda} e_j$ of ${\bigcap}_{j \in \Lambda} S_j$ at $p$.
	\item The intersection of the images of all differentials of the embeddings in ${\{e_j\}}_{j \in \Lambda}$ at $p$.
\end{enumerate} 
\end{enumerate}
\end{enumerate}
Then $D$ is said to be a {\it normal and convenient domain}. We also call it an {\it NCD}. 
\end{Def}
\begin{Rem}
\label{rem:1}
"Normal and convenient domain" or "NCD" seems to be identical to our paper, in addition to the withdrawn paper of us \cite{kitazawa8}. 
\end{Rem}
\begin{Rem}
	\label{rem:2}
	In Definition \ref{def:2}, in this case of non-singular hypersurfaces $S_j$, so-called normal vectors at each point in each intersection can be chosen as ones which are mutually independent. Of course, the number of the mutually independent normal vectors are same as the number of the distinct hypersurfaces intersecting there.
	Our transversality can be also discussed in this way.
\end{Rem}
The following reviews main ingredients of \cite{kitazawa3, kitazawa7, kitazawa8}.
\begin{Prop}
	\label{prop:1}
For an NCD $D \subset {\mathbb{R}}^n$ and an arbitrary sufficiently large integer $m$, we have some $m$-dimensional non-singular real algebraic connected manifold $M$ being also a connected component of the real algebraic set defined by some $l$ real polynomials and a smooth real algebraic map $f_D:M \rightarrow {\mathbb{R}}^n$ enjoying the following properties.
\begin{enumerate}
\item The image $f_D(M)$ is the closure $\overline{D}$ and $f_D(S(f_D))=\overline{D}-D$.
\item For each point $p \in {\mathbb{R}}^n$, the preimage ${f_D}^{-1}(p)$ is empty or diffeomorphic to the product of finitely many unit spheres and at most {\rm (}$m-n${\rm )}-dimensional. 
\end{enumerate} 
\end{Prop}

\begin{proof}

This essentially reviews a main result or Main Theorem 1 of \cite{kitazawa7} and \cite{kitazawa8}.

For example, we abuse the notation such as one from Definition \ref{def:2}.

For local coordinates and points for example, we use the notation like $x:=(x_1,\cdots,x_k)$ where $k>0$ is an arbitrary positive integer.

We define a set $S:= \{(x,y) \in \overline{D} \times {\mathbb{R}}^{m-n+l} \subset {\mathbb{R}}^n \times {\mathbb{R}}^{m-n+l}={\mathbb{R}}^{m+l} \mid f_j(x_1,\cdots,x_n)-||y_j||^2=0, j \in {\mathbb{N}_l}\}$ where $y_j:=(y_{j,1},\cdots,y_{j,d_j}) \in {\mathbb{R}}^{d_j}$ and $y:=(y_1,\cdots,y_l)$ with a suitable positive integer $d_j>0$.


 
We investigate the function defined canonically from the real polynomial $f_j(x_1,\cdots,x_n)-{\Sigma}_{j^{\prime}=1}^{d_j} {y_{j,j^{\prime}}}^2$. We investigate its partial derivative for variants $x_j$ and $y_{j,j^{\prime}}$. We also show that we can apply implicit function theorem to each point of the set $S$ to show that $S$ is non-singular. \\
\ \\
Case 1. \quad The case of a point $(x_0,y_0) \in S$ such that $y_0$ is not the origin. \\
\indent By the assumption $f_{j}(x_0)>0$ for each $j \in {\mathbb{N}}_l$. 
Here we calculate the partial derivative of the function for each variant $y_{j,j_0}$, at $(x_0,y_0)$.
For our notation on $y_0$, we introduce $y_0:=(y_{0,1},\cdots y_{0,l})$ and $y_{0,j}:=(y_{0,j,1},\cdots,y_{0,j,d_j}) \in {\mathbb{R}}^{d_j}$. 
As a result, the value is $2{y_{j,j_{0,0}}}=2y_{0,j,j_{0,0}} \neq 0$ for some variant $y_{j,j_{0,0}}$ with $j_0=j_{0,0}$. 
We calculate the partial derivative of the function for each variant $y_{j^{\prime},{j_0}^{\prime}}$ satisfying $j^{\prime} \neq j$. The value is $2y_{j^{\prime},{j_0}^{\prime}}=2y_{0,j^{\prime},{j_0}^{\prime}}=0$ for any variant $y_{j^{\prime},{j_0}^{\prime}}$ satisfying $j^{\prime} \neq j$. 

Based on this, we explain about the map into ${\mathbb{R}}^l$ obtained canonically from the family of the $l$ functions defined from the real polynomials.
The differential of the restriction of the function defined canonically from the real polynomial is not of rank $0$. This is not a singular point of this function defined from this polynomial. The map into ${\mathbb{R}}^l$ obtained canonically from the family of such $l$ functions is of rank $l$. \\
\ \\
Case 2. \quad The case of a point $(x_{\rm O},y_{\rm O}) \in S$ such that $y_{\rm O}$ is the origin.  \\
\indent For our notation on $y_{\rm O}$, we introduce $y_{\rm O}:=(y_{{\rm O},1},\cdots y_{{\rm O},l})$ and $y_{{\rm O},j}:=(y_{{\rm O},j,1},\cdots,y_{{\rm O},j,d_j}) \in {\mathbb{R}}^{d_j}$. In other words, "0" in "$y_0$" of Case 1 is changed into "${\rm O}$".
We calculate the partial derivative of the function for each variant $y_{j^{\prime},{j_0}^{\prime}}$ satisfying $j^{\prime} \neq j$, at $(x_{\rm O},y_{\rm O})$. The value is $2y_{j^{\prime},{j_0}^{\prime}}=2y_{{\rm O},j^{\prime},{j_0}^{\prime}}=0$ for any variant $y_{j^{\prime},{j_0}^{\prime}}$ satisfying $j^{\prime} \neq j$.

We may suppose the existence on some non-empty subset $J$ with the property that $x_{\rm O} \in S_{j}$ for each $j \in J \subset {\mathbb{N}}_l$ and that $x_{\rm O} \notin S_{j}$ for each $j \notin J$.

By the assumption, for the real polynomial $f_{j_1}(x)$, $f_{j_1}(x_{\rm O})> 0$ for each $j_1 \notin J$. Here we calculate the partial derivative of the function defined from the real polynomial $f_{j_J}(x)-{||y_{j_J}||}^2$ for each $j_J \in J$ and each variant $x_{j}$ ($1 \leq j \leq n$). 
 The function defined canonically from the real polynomial $f_{j_J}$ has no singular points on the hypersurface $S_{j_J}$ for each $j_J \in J$. This is since $S_{j}$ is non-singular for all $j \in {\mathbb{N}}_l$. From this, the partial derivative of the function defined from the real polynomial $f_{j_J}(x)-{||y_{j_J}||}^2$ is not $0$ for some variant $x_{j_{J,{\rm O}}}$ ($1 \leq j_{J,{\rm O}} \leq n$) with $j_J \in J$.
 
  We explain about the map into ${\mathbb{R}}^{|J|}$ obtained canonically from the $|J|$ real polynomials corresponding canonically to the $|J|$ variants. The rank of the differential at the point is $|J|$. This comes from the assumption on the "transversality" for the hypersurfaces $S_j$ in Definition \ref{def:1}.

We consider the canonically obtained map into ${\mathbb{R}}^{l-|J|}$ similarly respecting the remaining real polynomials. The rank of the differential at the point is $l-|J|$. We can know this from the fact that for each of these $l-|J|$ polynomials the partial derivative by some variant $y_{j_1,{j_{1,0}}^{\prime}}$ satisfying $j_1 \notin J$ is not $0$. We can also have such a case for the suitably and uniquely chosen $j_1 \notin J$ and some suitably chosen integer $1 \leq {j_{1,0}}^{\prime} \leq d_{j_1}$ which may not be unique in general. We cannot have such a case for any integer $j_1 \in {\mathbb{N}}_{l}-J$ and any integer $1 \leq {j_{1,0}}^{\prime} \leq d_{j_1}$. This also comes from an essential argument in Case 1. This is also presented in the beginning of our argument of Case 2 essentially.

Integrating these arguments here, we have a result as in Case 1. \\
\ \\
By our assumption and definition, for each point $x^{\prime}$ in ${\mathbb{R}}^n-\overline{D}$ sufficiently close to $\overline{D}$, we can define the set $S_{x^{\prime}}:= \{(x^{\prime},y) \in {\mathbb{R}}^n \times {\mathbb{R}}^{m-n+l} \subset {\mathbb{R}}^n \times {\mathbb{R}}^{m-n+l}={\mathbb{R}}^{m+l} \mid f_j(x_1,\cdots,x_n)-||y_j||^2=0, j \in {\mathbb{N}_l}\}$. We can also see that it is empty. This with implicit function theorem yields the fact that the set $S$ is an $m$-dimensional smooth closed and connected manifold and a non-singular real algebraic manifold. It is also a connected component of the real algebraic set defined by the $l$ real polynomials. 

We can also put $M:=S$ and define $f_D:M \rightarrow {\mathbb{R}}^n$ as the restriction of the canonical projection ${\pi}_{m+l,n}$. This completes the proof.
	
\end{proof}
\subsection{Some important NCDs.}
The complementary set of an affine subspace $L_{{\rm P},p_1,p_2} \subset {\mathbb{R}}^2$ in ${\mathbb{R}}^2$ consists of exactly two connected components. One of this is regarded as a connected component represented by either of the following two forms.
\begin{itemize}
\item $\{(x_1,x_2) \in {\mathbb{R}}^2 \mid x_2 \geq a_1 x_1+a_2\}$ where $a_1,a_2 \in \mathbb{R}$.
\item $\{(x_1,x_2) \in {\mathbb{R}}^2 \mid x_1 \geq a \}$ where $a \in \mathbb{R}$.
\end{itemize} 
Let $L_{{\rm P}_{+},p_1,p_2} \subset {\mathbb{R}}^2$ denote the closure of the connected component. 
Let $L_{{\rm P}_{-},p_1,p_2} \subset {\mathbb{R}}^2$ denote the closure of the other connected component.
The complementary set of $C_{{\rm H},(-\infty,a],b,c}$ or $C_{{\rm H},[a,\infty),b,c}$ in ${\mathbb{R}}^2$ consists of exactly two connected components. One of them is surrounded by $C_{{\rm H}_{+},(-\infty,a],b,c}:=\{(x_1,x_2) \in C_{{\rm H},(-\infty,a],b,c} \mid x_1>a, x_2>b\}$ and $C_{{\rm H}_{-},(-\infty,a],b,c}:=\{(x_1,x_2) \in C_{{\rm H},(-\infty,a],b,c} \mid x_1<a, x_2<b\}$ or $C_{{\rm H}_{+},[a,\infty),b,c}:=\{(x_1,x_2) \in C_{{\rm H},[a,\infty),b,c} \mid x_1<a, x_2>b\}$ and $C_{{\rm H}_{-},[a,\infty),b,c}:=\{(x_1,x_2) \in C_{{\rm H},[a,\infty),b,c} \mid x_1>a, x_2<b\}$ according to the chosen hyperbola $C_{{\rm H},(-\infty,a],b,c}$ or $C_{{\rm H},[a,\infty),b,c}$. Let $R_{{\rm H},(-\infty,a],b,c}$ denote the closure of the connected component for $C_{{\rm H},(-\infty,a],b,c}$. Let $R_{{\rm H},[a,\infty),b,c}$ denote the closure of the connected component for $C_{{\rm H},[a,\infty),b,c}$. See FIGURE \ref{fig:1}, showing an example for $b=0$, for example.

\begin{figure}
	
	\includegraphics[height=75mm, width=100mm]{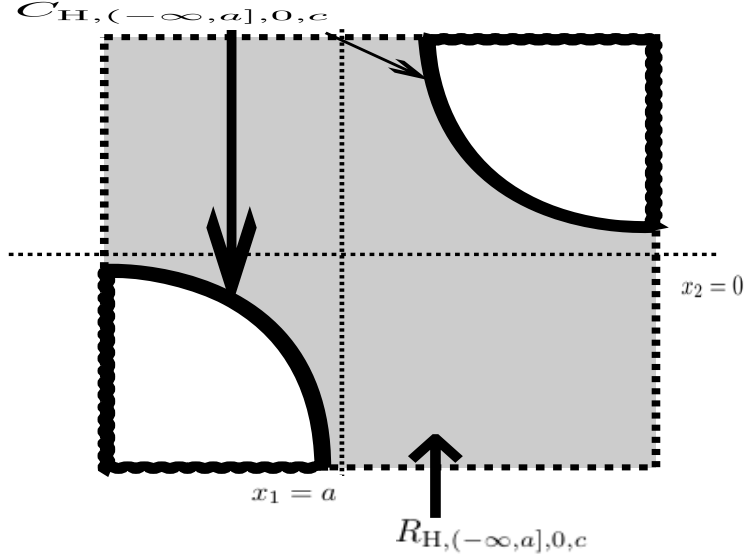}

	\caption{$R_{{\rm H},(-\infty,a],0,c}$ and some important (sub)sets of ${\mathbb{R}}^2$.}
	\label{fig:1}
\end{figure}
We can introduce the following subsets in ${\mathbb{R}}^n$ as the closures of explicit NCDs. They are important in our arguments.
\subsubsection{$R_{+,\{s_{1,j}\}_{j=1}^3,s_2,s}$ and $R_{-,\{s_{1,j}\}_{j=1}^3,s_2,s}$.}
Let $s_{1,1}<s_{1,2}<s_{1,3}$ and $s_2,s>0$ be real numbers. We also assume the following.

\begin{itemize}
\item $(s_{1,3},s_2) \in C_{{\rm H},(-\infty,s_{1,2}],0,s}$ and $(s_{1,3},s_2) \in C_{{\rm H}_{+},(-\infty,s_{1,2}],0,s}$.
\item There exists a unique negative number ${s_2}^{\prime}<0$ and $(s_{1,1},{s_2}^{\prime}) \in C_{{\rm H},(-\infty,s_{1,2}],0,s}$ and $(s_{1,1},{s_2}^{\prime}) \in C_{{\rm H}_{-},(-\infty,s_{1,2}],0,s}$ hold.
\end{itemize}
We define $R_{+,\{s_{1,j}\}_{j=1}^3,s_2,s}$ as the intersection \\
$L_{{\rm P}_{+},(s_{1,1},0),(s_{1,1},1)} \bigcap L_{{\rm P}_{+},(s_{1,1},0),(s_{1,2},0)} \bigcap L_{{\rm P}_{+},(s_{1,2},0),(s_{1,3},s_2)} \bigcap R_{{\rm H},(-\infty,s_{1,2}],0,s}$.

As before, let $s_{1,1}<s_{1,2}<s_{1,3}$ and $s_2, s>0$ be real numbers. We also assume the following.

\begin{itemize}
\item $(s_{1,1},s_2) \in C_{{\rm H},[s_{1,2},\infty),0,-s}$ and $(s_{1,1},s_2) \in C_{{\rm H}_{+},[s_{1,2},\infty),0,-s}$.
\item There exists a unique negative number ${s_2}^{\prime}<0$ and $(s_{1,3},{s_2}^{\prime}) \in C_{{\rm H},[s_{1,2},\infty),0,-s}$ and $(s_{1,3},{s_2}^{\prime}) \in C_{{\rm H}_{-},[s_{1,2},\infty),0,-s}$ hold.
\end{itemize}
We define $R_{-,\{s_{1,j}\}_{j=1}^3,s_2,s}$ as the intersection \\
$L_{{\rm P}_{-},(s_{1,3},0),(s_{1,3},1)} \bigcap L_{{\rm P}_{+},(s_{1,2},0),(s_{1,3},0)} \bigcap L_{{\rm P}_{+},(s_{1,2},0),(s_{1,1},s_2)} \bigcap R_{{\rm H},[s_{1,2},\infty),0,-s}$.
FIGURE 1 of \cite{kitazawa7} shows a subset very similar to $R_{-,\{s_{1,j}\}_{j=1}^3,s_2,s}$ in an explicit way. In our definition the thick curve in the parabola there is changed into a straight segment passing through the boundary point in the left of the connected component of the hyperbola in the left. 
\subsubsection{$R_{\{s_{1,j}\}_{j=1}^2,s_2,s}$.}
Let $s_{1,1}<s_{1,2}$ and $s_2,s>0$ be real numbers. 
We can define $R_{\{s_{1,j}\}_{j=1}^2,s_2,s}$ as the intersection
$L_{{\rm P}_{+},(s_{1,1},0),(s_{1,2},0)} \bigcap R_{{\rm H},[s_{1,1},\infty),0,-s} \bigcap R_{{\rm H},(-\infty,s_{1,2}],0,s}$ of the three disjoint non-singular connected curves.

\subsection{A proof of Main Theorems.}

\begin{MainThm}
\label{mthm:2}	
In Main Theorem \ref{mthm:1}, we can have the function as the composition of a suitable smooth real algebraic map into ${\mathbb{R}}^n$ enjoying the following properties with the canonical projection ${\pi}_{n,1}:{\mathbb{R}}^n \rightarrow \mathbb{R}$, mapping $(x_1,\cdots x_n)$ to $x_1$.
\begin{enumerate}
	\item The image of the map is the closure $\overline{D}$ of an NCD $D$ where we abuse the notation used in Definition \ref{def:2} for example.
	\item $\overline{D}-D$ is a union of finitely many {\rm (}$n-1${\rm )}-dimensional smooth connected submanifolds of $S_j$ as in Definition \ref{def:2}. Each $S_j$ is a connected component of the real algebraic set defined by a single real polynomial of degree at most $2$. More precisely, The family $\{S_j\}$ of hypersurfaces satisfies the following conditions.
	\begin{enumerate}
		\item It is an affine subspace of ${\mathbb{R}}^n$, a connected component of an {\rm (}$n-1${\rm )}-dimensional cylinder of a hyperbola in ${\mathbb{R}}^n$, or the boundary of some $n$-dimensional ellipsoid in ${\mathbb{R}}^n$.
		\item If $S_{j_0}$ is the boundary of some $n$-dimensional ellipsoid in ${\mathbb{R}}^n$, then for any other $S_j$ satisfying $j \neq j_0$, $S_j \bigcap S_{j_0}$ is empty.
	\end{enumerate}
\end{enumerate}
\end{MainThm}

\begin{proof}[A proof of Main Theorems]
We obtain a desired NCD $D$ as follows. \\
\ \\
Case 1. The case $l=2$. \\
In the case $l_{{\mathbb{N}}_{1}}(1)=0$, we put $n=2$ and define $D$ as the interior of $$L_{{\rm P}_{+},(t_1,0),(t_1,1)} \bigcap L_{{\rm P}_{+},(t_1,0),(t_2,0)} \bigcap L_{{\rm P}_{-},(t_2,0),(t_2,1)} \bigcap L_{{\rm P}_{-},(t_1,1),(t_2,1)}$$ and
in the case $l_{{\mathbb{N}}_{1}}(1)=1$, we put $n=2$ and define $D$ as the interior of $$L_{{\rm P}_{+},(t_1,0),(t_1,1)} \bigcap L_{{\rm P}_{+},(t_1,0),(t_2,0)} \bigcap L_{{\rm P}_{-},(t_2,0),(t_2,1)}$$ to complete the proof. \\
\ \\
Case 2. The case $l=3$. \\
We put $n=2$ and define $D$ as the interior of either of the following four to complete the proof. 
\begin{itemize}
\item $L_{{\rm P}_{+},(t_1,0),(t_1,1)} \bigcap L_{{\rm P}_{+},(t_1,0),(t_2,0)} \bigcap L_{{\rm P}_{+},(t_2,0),(t_3,1)} \bigcap L_{{\rm P}_{-},(t_1,1),(t_3,1)}$ in the case $l_{{\mathbb{N}}_{2}}(j)=0$ for each $j$.
\item $L_{{\rm P}_{+},(t_1,0),(t_1,1)} \bigcap L_{{\rm P}_{+},(t_1,1),(t_2,0)} \bigcap L_{{\rm P}_{+},(t_2,0),(t_3,1)} \bigcap L_{{\rm P}_{-},(t_3,0),(t_3,1)}$ in the case $l_{{\mathbb{N}}_{2}}(j)=1$ for each $j$.
\item $R_{+,\{t_j\}_{j=1}^3,0,1}$ in the case $l_{{\mathbb{N}}_{2}}(1)=1$ and $l_{{\mathbb{N}}_{2}}(2)=0$.
\item $R_{-,\{t_j\}_{j=1}^3,0,1}$ in the case $l_{{\mathbb{N}}_{2}}(1)=0$ and $l_{{\mathbb{N}}_{2}}(2)=1$.
\end{itemize}
\ \\
\ \\
Case 3. Remaining general cases $l \geq 4$. \\
Here let $\{t_{i(j)}\}_{j=1}^{l^{\prime}} \subset \{t_j\}_{j=1}^l$ be a subsequence of numbers consisting of all numbers $i(j) \in {\mathbb{N}}_{l-1}$ such that $l_{{\mathbb{N}}_{l-1}}(i(j))=1$. $l^{\prime}$ is a non-negative integer which may be $0$. In the case $l^{\prime}=0$, the sequence is empty. \\
\ \\

We define $n$ and $D$ as follows.
We put $D_0$ as the interior of $$L_{{\rm P}_{+},(t_1,0),(t_1,1)} \bigcap L_{{\rm P}_{+},(t_1,0),(t_{l},0)} \bigcap L_{{\rm P}_{-},(t_{l},0),(t_{l},1)} \bigcap L_{{\rm P}_{-},(t_1,1),(t_{l},1)} \subset {\mathbb{R}}^2$$ and we can see that $D_0$ is an NCD. We can define
 $D_1$ as an open set of ${\mathbb{R}}^2 \times {\mathbb{R}}^{l^{\prime}+1}=\mathbb{R} \times {\mathbb{R}}^{l^{\prime}+2}={\mathbb{R}}^n$ enjoying the following rule. Here we put $t_{j,1}:=t_{i(j)}$ and $t_{j,2}:=t_{i(j)+1}$.

Let $x:=(x_1,\cdots,x_{l^{\prime}+3}) \in {\mathbb{R}}^{l^{\prime}+3}$. $x \in D_1$ if and only if the following three hold. 
\begin{itemize}
\item $(x_1,x_2) \in D_0$.
\item $(x_1,x_{j+2})$ is in the interior of $R_{\{t_{j,j^{\prime}}\}_{j^{\prime}=1}^2,0,1}$ for each $j \in {\mathbb{N}}_{l^{\prime}}$.
\item $-R<x_{l^{\prime}+3}<R$ for some positive real number $R$.
\end{itemize}

For each integer $1 \leq j \leq l-3$, we choose a suitable and sufficiently small $B_{{\rm E},\{a_{j,j^{\prime}}\}_{j^{\prime}=1}^{l^{\prime}+3},\{r_{j,j^{\prime}}\}_{j^{\prime}=1}^{l^{\prime+3}}}$ satisfying $a_{j,1}=\frac{t_{j+1}+t_{j+2}}{2}$ and $r_{j,1}=\frac{{(t_{j+2}-t_{j+1})}^2}{4}$. We can choose these $n$-dimensional ellipsoids in ${\mathbb{R}}^n$ disjointly in the interior of $D_1$ and we do. We remove the interiors. Our desired set $D$ is the interior of the resulting set.

\ \\

Proposition \ref{prop:1} yields our desired map $f_D:M \rightarrow {\mathbb{R}}^n$.
We explain about some important remarks. 
We can know types of each resulting hypersurface $S_j$ by the construction. We can consider the natural connections on Euclidean spaces as the natural Riemannian manifolds and consider the notion of (mutually) "{\it parallel}" objects such as {\it parallel} tangent vectors and {\it parallel} subsets for example. The following list shows all of these hypersurfaces $S_j$ here.
Here let $e_i \in {\mathbb{R}}^{l^{\prime}+3}$ denote the tangent vector at the origin $0 \in {\mathbb{R}}^{l^{\prime}+3}$ represented as the vector whose $i$-th component is $1$ and whose $i^{\prime}$-th component is $0$ for any integer $i^{\prime} \neq i$ satisfying $1 \leq i^{\prime} \leq l^{\prime}+3$.  
\begin{itemize}
	\item Two mutually disjoint affine subspaces parallel to the affine subspace $\{0\} \times {\mathbb{R}}^{l^{\prime}+2}$. At each point there, each normal vector is regarded to be parallel to some vector represented by the form $t e_1$ with some number $t \neq 0$.
	\item Two mutually disjoint affine subspaces parallel to the affine subspace $\mathbb{R} \times \{0\} \times {\mathbb{R}}^{l^{\prime}+1}$. At each point there, each normal vector is regarded to be parallel to some vector represented by the form $t e_2$ with some number $t \neq 0$.
	\item Two mutually disjoint affine subspaces parallel to the affine subspace ${\mathbb{R}}^{l^{\prime}+2} \times \{0\}$. At each point there, each normal vector is regarded to be parallel to some vector represented by the form $t e_{l^{\prime}+3}$ with some number $t \neq 0$.
	\item Exactly $l^{\prime}$ affine subspaces. The $i$-th affine subspace in these $l^{\prime}$ affine subspaces is equal to ${\mathbb{R}}^{i+1} \times \{0\} \times {\mathbb{R}}^{l^{\prime}-i+1}$. At each point there, each normal vector is regarded to be parallel to some vector represented by the form $t e_{i+2}$ with some number $t \neq 0$.
	\item Exactly $2l^{\prime}$ cylinders of hyperbolas.
	For each integer $1 \leq i \leq l^{\prime}$, the ($2i-1$)-th hypersurface and the ($2i$)-th one here are mutually disjoint. These two are represented as subsets of ${\mathbb{R}}^{l^{\prime}+3}$ obtained in the following steps.
	\begin{itemize}
		\item We choose two subsets $C_{{\rm H}_{+},[t_{i,1},\infty),0,1}$ and $C_{{\rm H}_{+},(
-\infty,t_{i,2}],0,1}$. Consider the product of each subset and ${\mathbb{R}}^{l^{\prime}+1}$, This new subset is a subset of ${\mathbb{R}}^2 \times {\mathbb{R}}^{l^{\prime}+1}={\mathbb{R}}^{l^{\prime}+3}$. 
		\item We map the previously obtained subsets by an affine transformation $\sigma:{\mathbb{R}}^{l^{\prime}+3} \rightarrow {\mathbb{R}}^{l^{\prime}+3}$ defined uniquely in the following way.
		\begin{itemize}
			\item For each $x \in {\mathbb{R}}^{l^{\prime}+3}$, the 1st component of $\sigma(x)$ does not change under the transformation $\sigma$ and it is equal to $x_1$.
			\item For each $x \in {\mathbb{R}}^{l^{\prime}+3}$, the $j$-th component of $\sigma(x)$ may change under the transformation $\sigma$ and equal to $x_{j+1}$ for $2 \leq j \leq i+1$.
			\item For each $x \in {\mathbb{R}}^{l^{\prime}+3}$, the ($i+2$)-th component of $\sigma(x)$ may change under the transformation $\sigma$ and equal to $x_{2}$.
				\item For each $x \in {\mathbb{R}}^{l^{\prime}+3}$, the $j$-th component of $\sigma(x)$ does not change under the transformation $\sigma$ and equal to $x_{j}$ for $i+3 \leq j \leq l^{\prime}+3$.
		\end{itemize}
	At each point there, each normal vector is regarded to be parallel to some vector represented by the form $t_{i,1}e_{1}+t_{i,2}e_{i+2}$ with some numbers $t_{i,1} \neq 0$ and $t_{i,2} \neq 0$.
	Furthermore, these two are also apart from the $i$-th affine subspace equal to ${\mathbb{R}}^{i+1} \times \{0\} \times {\mathbb{R}}^{l^{\prime}-i+1}$ presented before.
	\end{itemize}
	\item The boundaries of some $n$-dimensional ellipsoids in ${\mathbb{R}}^n$. The $n$-dimensional ellipsoids in ${\mathbb{R}}^n$ are mutually disjoint and they are also apart from the other hypersurfaces presented here. 
\end{itemize}

We discuss transversality. Observe the list and see the normal vectors explicitly. Recall also Remark \ref{rem:2} for example. We can see that the transversality is satisfied. 

We discuss singularities of the function $f$.
For the boundary of each $n$-dimensional ellipsoid in ${\mathbb{R}}^n$, only points regarded as the two "poles" are regarded as points in this connected component of $\overline{D}-D$ whose preimages (for our map $f_D$) contain some singular points of the function $f:={\pi}_{n,1} \circ f_D$. On singular points of the function $f$, we can explain similarly about points in the two real affine subspaces in ${\mathbb{R}}^{l^{\prime}+3}$ parallel to $\{0\} \times {\mathbb{R}}^{l^{\prime}+2}$ or the set $\{t_1,t_l\} \times {\mathbb{R}}^{l^{\prime}+2}$.
In the image $\overline{D} \subset {\mathbb{R}}^n$ of the map $f_D$, except such points, the preimages contain no singular points of $f$.

This completes the proof.
	\end{proof}

We can have another similar result. 
\begin{MainThm}
	\label{mthm:3}
	In Main Theorem \ref{mthm:1}, let $l \neq 3$. Instead of Main Theorem \ref{mthm:2}, we can have the function as the composition of a suitable smooth real algebraic map into ${\mathbb{R}}^n$ enjoying the following properties with the canonical projection ${\pi}_{n,1}:{\mathbb{R}}^n \rightarrow \mathbb{R}$, mapping $(x_1,\cdots x_n)$ to $x_1$.
	\begin{enumerate}
		\item The image of the map is the closure $\overline{D}$ of an NCD $D$ where we abuse the notation used in Definition \ref{def:2} for example.
		\item $\overline{D}-D$ is a union of finitely many {\rm (}$n-1${\rm )}-dimensional smooth connected submanifolds of $S_j$ as in Definition \ref{def:2}. Each $S_j$ is a connected component of the real algebraic set defined by a single real polynomial of degree at most $2$. More precisely, The family $\{S_j\}$ of hypersurfaces satisfies the following conditions.
		\begin{enumerate}
			\item It is an affine subspace of ${\mathbb{R}}^n$, a connected component of an {\rm (}$n-1${\rm )}-dimensional cylinder of a hyperbola in ${\mathbb{R}}^n$, the boundary of some $n$-dimensional ellipsoid in ${\mathbb{R}}^n$, or the boundary of some $n$-dimensional cylinder of some ellipsoid in ${\mathbb{R}}^n$ represented as the product of a 2-dimensional ellipsoid in ${\mathbb{R}}^2$ and ${\mathbb{R}}^{n-2}$.
			\item If $l_{{\mathbb{N}}_{l-1}}({\mathbb{N}}_{l-1})=\{0\}$, then each $S_{j}$ is the boundary of some $n$-dimensional ellipsoid and these hypersurfaces are mutually disjoint.
			\item If $l_{{\mathbb{N}}_{l-1}}({\mathbb{N}}_{l-1})=\{0,1\}$, then each hypersurface $S_{j_0}$ which is the boundary of some $n$-dimensional ellipsoid and any hypersurface $S_{j}$ satisfying $j \neq j_0$ are mutually disjoint.
		\end{enumerate}
	\end{enumerate}
	\end{MainThm}

	\begin{proof}[A proof]
		We obtain a desired NCD $D$ as follows. \\
		\ \\
		Case 1. The case $l=2$. \\
		In the case $l_{{\mathbb{N}}_{1}}(1)=0$, we put $n=2$ and define $D$ as the interior of
		$B_{{\rm E},\{a_{0,j^{\prime}}\}_{j^{\prime}=1}^{2},\{r_{0,j^{\prime}}\}_{j^{\prime}=1}^{2}}$ satisfying $a_{0,1}=\frac{t_1+t_2}{2}$ and $r_{j,1}=\frac{{(t_2-t_1)}^2}{4}$.
		
		\noindent In the case $l_{{\mathbb{N}}_{1}}(1)=1$, we put $n=2$ and define $D$ as the interior of
		$B_{{\rm E},\{\frac{t_1+t_2}{2}\},\{\frac{{(t_2-t_1)}^2}{4}\}} \times \mathbb{R} \subset {\mathbb{R}}^2$. \\
		\ \\
		Case 2. Remaining general cases $l \geq 4$. \\
		In the case $l_{{\mathbb{N}}_{l-1}}(
		{\mathbb{N}}_{l-1}
		)=\{0\}$, we put $n=3$ and define $D$ as the interior of the set obtained in the following way.
		\begin{itemize}
			\item We choose $B_{{\rm E},\{a_{0,j^{\prime}}\}_{j^{\prime}=1}^{3},\{r_{0,j^{\prime}}\}_{j^{\prime}=1}^{3}}$
			 satisfying $a_{0,1}=\frac{t_1+t_l}{2}$ and $r_{0,1}=\frac{{(t_l-t_1)}^2}{4}$.
			\item We choose the previous set as a sufficiently large set. We can choose a family $\{B_{{\rm E},\{a_{j,j^{\prime}}\}_{j^{\prime}=1}^{3},\{r_{j,j^{\prime}}\}_{j^{\prime}=1}^{3}}\}$ of $n$-dimensional ellipsoids in ${\mathbb{R}^n}$ satisfying $a_{j,1}=\frac{t_
				{j+1}+
				t_{j+2}}{2}$ and $r_{j,1}=\frac{{(t_{j+2}-t_{j+1}	)}^2}{4}$ in the interior disjointly for $1 \leq j \leq l-3$. We remove their interiors from $B_{{\rm E},\{a_{0,j^{\prime}}\}_{j^{\prime}=1}^{3},\{r_{0,j^{\prime}}\}_{j^{\prime}=1}^{3}}$.
		\end{itemize}
		
		
	We consider the case
$l_{{\mathbb{N}}_{
l-1}}({\mathbb{N}}_{
l-1})=\{0,1\}$. Here, as in the proof of Main Theorem \ref{mthm:2}, let $\{t_{i(j)}\}_{j=1}^{l^{\prime}} \subset \{t_j\}_{j=1}^l$ be a subsequence of numbers consisting of all numbers $i(j) \in {\mathbb{N}}_{l-1}$ such that $l_{{\mathbb{N}}_{l-1}}(i(j))=1$. $l^{\prime}$ is a positive integer. \\
		\ \\
		
		We define $n$ and $D$ as follows.
		We put $D_0$ as the interior of
		$B_{{\rm E},\{a_{0,j^{\prime}}\}_{j^{\prime}=1}^{2},\{r_{0,j^{\prime}}\}_{j^{\prime}=1}^{2}}$ satisfying $a_{0,1}=\frac{t_1+t_l}{2}$ and $r_{0,1}=\frac{{(t_l-t_1)}^2}{4}$ and we can see that $D_0$ is an NCD. We can define
		$D_1$ as an open set of ${\mathbb{R}}^2 \times {\mathbb{R}}^{l^{\prime}+1}=\mathbb{R} \times {\mathbb{R}}^{l^{\prime}+2}={\mathbb{R}}^n$ enjoying the following rule. Here we put $t_{j,1}:=t_{i(j)}$ and $t_{j,2}:=t_{i(j)+1}$.
		
		Let $x:=(x_1,\cdots,x_{l^{\prime}+3}) \in {\mathbb{R}}^{l^{\prime}+3}$. $x \in D_1$ if and only if the following three hold. 
		\begin{itemize}
			\item $(x_1,x_2) \in D_0$.
			\item $(x_1,x_{j+2})$ is in the interior of $R_{\{t_{j,j^{\prime}}\}_{j^{\prime}=1}^2,0,1}$ for each $j \in {\mathbb{N}}_{l^{\prime}}$.
			\item $-R<x_{l^{\prime}+3}<R$ for some positive real number $R$.
		\end{itemize}
		
		For each integer $1 \leq j \leq l-3$, we choose a suitable and sufficiently small $B_{{\rm E},\{a_{j,j^{\prime}}\}_{j^{\prime}=1}^{l^{\prime}+3},\{r_{j,j^{\prime}}\}_{j^{\prime}=1}^{l^{\prime+3}}}$ satisfying $a_{j,1}=\frac{t_{j+1}+t_{j+2}}{2}$ and $r_{j,1}=\frac{{(t_{j+2}-t_{j+1})}^2}{4}$. We can choose these $n$-dimensional ellipsoids in ${\mathbb{R}}^n$ disjointly in the interior of $D_1$ and we do. We remove the interiors. Our desired set $D$ is the interior of the resulting set.
		
		\ \\

		Proposition \ref{prop:1} yields our desired map $f_D:M \rightarrow {\mathbb{R}}^n$.
		We explain about some important remarks. 
		We can know types of each resulting hypersurface $S_j$ by the construction. We can consider the natural connections on Euclidean spaces as the natural Riemannian manifolds as in the proof of Main Theorem \ref{mthm:2}. The following list shows all of these hypersurfaces $S_j$ here.
		As in the proof of Main Theorem \ref{mthm:2}, let $e_i \in {\mathbb{R}}^{l^{\prime}+3}$ denote the tangent vector at the origin $0 \in {\mathbb{R}}^{l^{\prime}+3}$ represented as the vector whose $i$-th component is $1$ and whose $i^{\prime}$-th component is $0$ for any integer $i^{\prime} \neq i$ satisfying $1 \leq i^{\prime} \leq l^{\prime}+3$.  
		\begin{itemize}
			\item The boundary of the ($l^{\prime}+3$)-dimensional cylinder $\partial \overline{D_0} \times {\mathbb{R}}^{l^{\prime}+1} \subset {\mathbb{R}}^{l^{\prime}+3}$ of the ellipsoid $\overline{D_0}$ in ${\mathbb{R}}^{l^{\prime}+3}$. At each point there, normal vectors are regarded to be parallel to some vector represented by the form $t_1 e_1+t_2 e_2$ with some pair $(t_1,t_2) \neq (0,0)$ of real numbers.
			\item Two mutually disjoint affine subspaces parallel to the affine subspace ${\mathbb{R}}^{l^{\prime}+2} \times \{0\}$. At each point there, normal vectors are regarded to be parallel to some vector represented by the form $t e_{l^{\prime}+3}$ with some number $t \neq 0$.
			\item Exactly $l^{\prime}$ affine subspaces. The $i$-th affine subspace in these $l^{\prime}$ affine subspaces is equal to ${\mathbb{R}}^{i+1} \times \{0\} \times {\mathbb{R}}^{l^{\prime}-i+1}$. At each point there, normal vectors are regarded to be parallel to some vector represented by the form $t e_{i+2}$ with some number $t \neq 0$.
			\item Exactly $2l^{\prime}$ cylinders of hyperbolas.
			For each integer $1 \leq i \leq l^{\prime}$, the ($2i-1$)-th hypersurface and the ($2i$)-th one here are mutually disjoint. These two are represented as subsets of ${\mathbb{R}}^{l^{\prime}+3}$ obtained in the following steps.
			\begin{itemize}
				\item We choose two subsets $C_{{\rm H}_{+},[t_{i,1},\infty),0,1}$ and $C_{{\rm H}_{+},(
					-\infty,t_{i,2}],0,1}$. Consider the product of each subset and ${\mathbb{R}}^{l^{\prime}+1}$, This new subset is a subset of ${\mathbb{R}}^2 \times {\mathbb{R}}^{l^{\prime}+1}={\mathbb{R}}^{l^{\prime}+3}$. 
				\item We map the previously obtained subsets by an affine transformation $\sigma:{\mathbb{R}}^{l^{\prime}+3} \rightarrow {\mathbb{R}}^{l^{\prime}+3}$ defined in the following way.
				\begin{itemize}
					\item For each $x \in {\mathbb{R}}^{l^{\prime}+3}$, the 1st component of $\sigma(x)$ does not change under the transformation $\sigma$ and it is equal to $x_1$.
					\item For each $x \in {\mathbb{R}}^{l^{\prime}+3}$, the $j$-th component of $\sigma(x)$ may change under the transformation $\sigma$ and equal to $x_{j+1}$ for $2 \leq j \leq i+1$.
					\item For each $x \in {\mathbb{R}}^{l^{\prime}+3}$, the ($i+2$)-th component of $\sigma(x)$ may change under the transformation $\sigma$ and equal to $x_{2}$.
					\item For each $x \in {\mathbb{R}}^{l^{\prime}+3}$, the $j$-th component of $\sigma(x)$ does not change under the transformation $\sigma$ and equal to $x_{j}$ for $i+3 \leq j \leq l^{\prime}+3$.
				\end{itemize}
				At each point there, each normal vector is regarded to be parallel to some vector represented by the form $t_{i,1}e_{1}+t_{i,2}e_{i+2}$ with some numbers $t_{i,1} \neq 0$ and $t_{i,2} \neq 0$.
				Furthermore, these two are also apart from the $i$-th affine subspace equal to ${\mathbb{R}}^{i+1} \times \{0\} \times {\mathbb{R}}^{l^{\prime}-i+1}$ presented before.
			\end{itemize}
			\item The boundaries of some $n$-dimensional ellipsoids in ${\mathbb{R}}^n$. The $n$-dimensional ellipsoids in ${\mathbb{R}}^n$ are mutually disjoint and they are also apart from the other hypersurfaces presented here. 
		\end{itemize}
		
		We discuss transversality. Observe the list and see the normal vectors explicitly. Recall also Remark \ref{rem:2} for example. We can see that the transversality is satisfied. 
		
		We discuss singularities of the function $f$.
		For the boundary of each $n$-dimensional ellipsoid in ${\mathbb{R}}^n$, only points regarded as the two "poles" are regarded as points in this connected component of $\overline{D}-D$ whose preimages (for our map $f_D$) contain some singular points of the function $f:={\pi}_{n,1} \circ f_D$. On singular points of the function $f$, we can explain similarly about points in $\{(t_1,a_{0,2}),(t_l,a_{0,2})\} \times {\mathbb{R}}^{l^{\prime}+1}$. This is a ($l^{\prime}+1$)-dimensional affine subspace in ${\mathbb{R}}^{l^{\prime}+3}$.
		In the image $\overline{D} \subset {\mathbb{R}}^n$ of the map $f_D$, except such points, the preimages contain no singular points of $f$.
		
		This completes the proof.
	
\end{proof}
\begin{Rem}
\label{rem:3}
	This paper shows a variant of \cite{kitazawa8}, which is withdrawn due to an improvement. Construction of the real algebraic map into ${\mathbb{R}}^3$ of Main Theorem 2 of \cite{kitazawa8} aims to construct smooth real algebraic maps on non-singular real algebraic manifolds whereas we only succeed construction in cases the following two are satisfied "in terms of our Main Theorem \ref{mthm:1}".
	\begin{itemize}
		\item $l_{{\mathbb{N}}_{l-1}}(j)=0$ for $1<j<l-1$.
		\item At least one of $l_{{\mathbb{N}}_{l-1}}(1)=0$ or $l_{{\mathbb{N}}_{l-1}}(1)=0$ holds.
		\end{itemize}
	More precisely, we can check the version arXiv:2303.14988v2 and "A Proof of Main Theorems". More precisely, in the case either  $l_{{\mathbb{N}}_{l-1}}(1)=0$ or $l_{{\mathbb{N}}_{l-1}}(1)=0$ holds, check a case in Case 8 which is presented through Case 3 or Case 4 first.
	
	In constructing nice maps into ${\mathbb{R}}^3$ like ones in our Main Theorem \ref{mthm:2} in more general cases, we need semi-algebraic sets or ({\it non-singular}) Nash manifolds there.
We also consider construction such that the images of maps are represented as the intersections of the closures of open subsets of specific types of fixed Euclidean spaces in \cite{kitazawa7}.
For main results of \cite{kitazawa8} and our present paper, it seems that we can have various explicit construction. 

For example, if we admit the hypersurfaces $S_j$ to be more general ones regarded as some connected components of the real algebraic sets defined by (single) real polynomials of degrees at most $2$, what can we know? See Problem \ref{prob:3}, presented in the last.
	\end{Rem}

\section{Appendices, mainly problems related to our study.}
We present some problems related to our study.

For example, it may be important to explicitly give definitions of notions on graphs. However, we omit them. We expect that maybe at least we have some elementary knowledge on graphs. See also \cite{kitazawa1, kitazawa2, kitazawa3, kitazawa5, kitazawa6, kitazawa7, kitazawa8} for example.
\begin{Prob}
\label{prob:2}
	In Main Theorems, can we construct the function $f$ whose preimages have prescribed topologies?
\end{Prob}
In the case of smooth functions, \cite{kitazawa1, kitazawa2} are important. \cite{kitazawa1} considers the case of smooth functions on $3$-dimensional closed and orientable manifolds. According to this, we can have a smooth function with mild singularities of some natural class whose Reeb graph is an arbitrary finite graph and connected components containing no singular points of whose preimages are prescribed closed and connected surfaces. The singularities are, in short, singularities of {\it Morse} functions, {\it fold} maps, which are higher dimensional versions of Morse functions, and compositions of such functions and maps. See \cite{golubitskyguillemin} for fundamental theory of singularities of Morse functions, fold maps, and more general smooth maps. \cite{kitazawa5} generalizes (some) results of \cite{kitazawa1} for Morse functions on closed manifolds of general dimensions. 
\cite{kitazawa2} considers cases where connected components of preimages containing no singular points for the functions and the manifolds have no boundaries and may not be closed. This also constructs smooth functions which are not real analytic.
\cite{saeki} is also a related paper, which is based on our informal discussions on \cite{kitazawa1}. It considers very general cases. It also constructs smooth functions which are not real analytic.

\cite{kitazawa6} studies a very explicit case for smooth real algebraic functions on closed manifolds such that preimages are empty or connected. More precisely, (connected components of) the preimages are empty, spheres, or manifolds represented as connected sums of products of two spheres if they contain no singular points of the function.

\begin{Prob}
\label{prob:3}
The images of our real algebraic maps into ${\mathbb{R}}^n$, constructed in Main Theorem \ref{mthm:2} for example, are regarded as the closures of some $n$-dimensional open sets in ${\mathbb{R}}^n$. They collapse to some cell (CW) complexes whose dimensions are lower.
Can we consider suitable classes of graphs, CW complexes or more generally, cell complexes for such phenomena? Can we study about (the closures of) open sets surrounded by non-singular real algebraic hypersurfaces of some nice classes and collapsing to such nice complexes. In other words, can we study about nice {\it arrangements} of real algebraic hypersurfaces forming (the closures of) the open sets? 	
\end{Prob}
For this, \cite{bodinpopescupampusorea} is one of related studies. This is on realization of domains surrounded by non-singular connected real algebraic curves in the affine space ${\mathbb{R}}^2$ and collapsing to given graphs. 
In short, they first find desired domains topologically or in the smooth category. Next they apply a kind of approximations by real polynomials. This also has motivated us to obtain our related pioneering result \cite{kitazawa3}. For related studies, see \cite{kohnpieneranestadrydellshapirosinnsoreatelen} and see also \cite{sorea1, sorea2} for example. 

\begin{Prob}
\label{prob:4}
In our case of smooth real algebraic functions and maps on non-singular real algebraic (or more generally, Nash) manifolds, 
can we have explicit and nice affirmative answers for more general (Reeb) graphs? Here, we also need explicit and nice conditions on (the topologies and the differentiable structures of) connected components of preimages containing no singular points.
\end{Prob}
This seems to be very difficult. At present, we must restrict (our classes of Reeb) graphs to very simple or explicit ones as we do in \cite{kitazawa3, kitazawa6, kitazawa7, kitazawa8}. For example, related to our comments in Remark \ref{rem:3}, let us admit the hypersurfaces $S_j$ to be more general ones regarded as some connected components of the real algebraic sets defined by some (single) real polynomials of degrees at most $2$. What do we have then?

\end{document}